\documentclass[11pt,reqno]{amsart}

\usepackage{graphics,graphicx,fontenc,mathabx}
\usepackage{epsfig,psfrag,manfnt}
\usepackage{bbm,subfigure,rotating,bm,float,mathdots,wasysym}
\usepackage[all,cmtip]{xy}
\usepackage{amssymb,amsmath,amsfonts,amsthm,color}

\makeatletter
\newcommand*\bigcdot{\mathpalette\bigcdot@{.5}}
\newcommand*\bigcdot@[2]{\mathbin{\vcenter{\hbox{\scalebox{#2}{$\m@th#1\bullet$}}}}}
\makeatother

\newcommand{\calL}{\mathcal{L}}

\newcommand{\mR}{\mathbb{R}}

\newcommand{\bbi}{\mathbf{i}}
\newcommand{\bbj}{\mathbf{j}}
\newcommand{\bbk}{\mathbf{k}}

\newcommand{\bbq}{\mathbf{q}}

\newcommand{\bbv}{\mathbf{v}}
\newcommand{\bbw}{\mathbf{w}}

\newcommand{\bbD}{\mathbf{D}}

\newtheorem{theorem}{Theorem}[section]
\newtheorem{lemma}[theorem]{Lemma}

\newtheorem{proposition}[theorem]{Proposition}

\theoremstyle{definition}

\theoremstyle{definition}

\theoremstyle{definition}

\theoremstyle{definition}

\begin{document}



\title[]{Eigenvector of a matrix in $SO_3(\mR)$}

\author[A. Sasane]{Amol Sasane}
\address{Department of Mathematics \\London School of Economics\\
    Houghton Street\\ London WC2A 2AE\\ United Kingdom}
\email{A.J.Sasane@lse.ac.uk}
\author[V. Ufnarovski]{Victor Ufnarovski}
\address{Department of Mathematics\\ Lund University \\ 
S\"olvegatan 18, 223 62 Lund\\ Sweden}
\email{ufn@maths.lth.se}

\begin{abstract}
Let $A=[a_{ij}]\in O_3(\mR)$. We give several different proofs of 
the fact that the vector
$$
V:=\left[\begin{array}{ccc}
\displaystyle \frac{1}{a_{23}+a_{32}} &  
\displaystyle \frac{1}{a_{13}+a_{31}} & 
\displaystyle \frac{1}{a_{12}+a_{21}}
\end{array}\right]^T,
$$
if it exists, is an eigenvector of $A$ corresponding to the eigenvalue
$1$.
\end{abstract}

\maketitle

\section{Introduction}

Let $A$ be a $3\times 3$ real matrix and suppose that we want to find
an eigenvector $V$ for $A.$ Every student learns an algorithm for
this, but is it possible to skip the toil, and write down $V$
explicitly in terms of $a_{ij}$?  For example, we can easily do this
for a matrix of rank $1.$ If $X$ is a nonzero column, then we can
simply take $V=X.$ Indeed, we know that $A=XY^T$ for some vector $Y$
and
$$
AX=XY^TX=X\langle Y,X\rangle=\langle Y,X\rangle X,
$$ 
where we have used that the $1\times 1$ matrix $Y^TX$ can be
identified with the inner product $\langle Y,X\rangle$.  Another
interesting example is when we consider skew-symmetric matrices:

\begin{theorem}
\label{thm:skew}
For any $3\times 3$ skew-symmetrical matrix
$$
Q=\left[\begin{array}{rrr}
  0 & -r & q \\
  r & 0  & -p \\
 -q & p  & 0 
 \end{array}\right]
$$
the vector
$$
V=\left[\begin{array}{r}
  p \\
  q \\
  r 
  \end{array}\right]
$$ 
belongs to its kernel, thus $QV=0.$
\end{theorem}

\noindent This can be checked directly, but in fact we can generalise
this to any matrix of rank $2.$

\begin{theorem}
\label{thm:rank2}
Let $A_{ij}=(-1)^{i+j}D_{ij}$ where $D_{ij}$ is a minor obtained by
deleting the row $i$ and column $j$ from the matrix $A.$ If A has rank
$2$, then all three vectors $ V_j=[A_{j1}\;\;A_{j2}\;\;A_{j3}]^T$
belong to its kernel and at least one of them is non-zero eigenvector.
\end{theorem}
\begin{proof}
It is well-known that (see for example \cite[Theorem~3.15,p.69]{HU})
$$
\sum_{k=1}^3 a_{ik}A_{jk}=\delta_{ij}\det A,
$$ 
where $\delta_{ij}$ is $1$ if $i=j$ and $0$ otherwise. 
In the case of rank $2$, we get that $\det A=0\Rightarrow AV_j=0$, 
and at least one of the vectors $V_j$ is non-zero.
\end{proof}

\noindent What can be said about non-singular matrices? If we know an
eigenvalue $\lambda$ we can simply apply the same arguments to the
matrix $A-\lambda I$ to find the eigenvector (the case $A=\lambda I$
will be special, but here we can take any non-zero vector). We always
know an eigenvalue $\pm 1$ for an orthogonal matrices. For example it
is well-known that $A\in SO_3(\mR)$ describes a rotation in $\mR^3$
about some axis described by a vector $V$ (see e.g. \cite[Thm. 5.5,
p.124]{Art}), and this $V$ is an eigenvector of $A$ corresponding to
the eigenvalue $1$. So we want to express axis of rotation in terms of
the matrix entries of $A$. But unexpectedly, we can get the vector $V$
quite easily.

\begin{theorem} 
\label{thm:general} 
Let $A=[a_{ij}] \in  SO_3(\mathbb{R})$. Let
\begin{eqnarray*}
 V&=&\left[\begin{array}{ccc}
     \displaystyle \frac{1}{a_{23}+a_{32}} & 
     \displaystyle \frac{1}{a_{13}+a_{31}} & 
     \displaystyle \frac{1}{a_{12}+a_{21}}
     \end{array}\right]^T,\\
 U&=&\left[\begin{array}{ccc}
     \displaystyle {a_{23}-a_{32}} & 
     \displaystyle {a_{31}-a_{13}} & 
     \displaystyle {a_{12}-a_{21}}
     \end{array}\right]^T,\phantom{\displaystyle \frac{1}{a_{32}}}\\
 W_1&=&\left[\begin{array}{ccc}
       \displaystyle {1+a_{11}-a_{22}-a_{33}} & 
       \displaystyle {a_{12}+a_{21}} & 
       \displaystyle {a_{13}+a_{31}}
       \end{array}\right]^T,\phantom{\displaystyle \frac{1}{a_{32}}}\\
 W_2&=&\left[\begin{array}{ccc}
       \displaystyle {a_{12}+a_{21}} & 
       \displaystyle {1+a_{22}-a_{11}-a_{33}} & 
       \displaystyle {a_{23}+a_{32}} 
       \end{array}\right]^T,\phantom{\displaystyle \frac{1}{a_{32}}}\\
 W_3&=&\left[\begin{array}{ccc}
       \displaystyle {a_{13}+a_{31}} & 
       \displaystyle {a_{23}+a_{32}} &
       \displaystyle {1+a_{33}-a_{11}-a_{22}} 
       \end{array}\right]^T.\phantom{\displaystyle \frac{1}{a_{32}}}
\end{eqnarray*}
Then $AV=V,AU=U, AW_i=W_i$, so any of these vectors $($if it exists
and is non-zero$)$, is an eigenvector with eigenvalue $1$. If
$A\neq I$ then at least one of them exists and is non-zero.
\end{theorem}

\noindent The most unexpected one is the vector $V$ so we concentrate
on it.

\begin{theorem}
\label{thm:main}
Let $A=[a_{ij}] \in  SO_3(\mathbb{R})$.
If the vector
$$
V=\left[\begin{array}{ccc}
  \displaystyle \frac{1}{a_{23}+a_{32}} & 
  \displaystyle \frac{1}{a_{13}+a_{31}} & 
  \displaystyle \frac{1}{a_{12}+a_{21}}
  \end{array}\right]^T
$$
exists $($that is, the denominators are non-zeros$)$, then $AV=V.$
\end{theorem}

\noindent In fact, this result appears as an exercise in M. Artin's
classic textbook {\em Algebra} \cite[Ex.14, \S5, Chap.4, p.149]{Art}.
Our plan is to give several different proofs of Theorem \ref{thm:main}
obtaining simultaneously the proof of Theorem \ref{thm:general}.

\medskip

\noindent {\bf Acknowledgement:} The authors thank their colleagues
Mikael Sundqvist and J\"org Schmeling for useful discussions.

\section{Two algebraic proofs}

We start from some useful statements.

\begin{theorem}
For arbitrary $n$ and any $A\in SO_n(\mathbb{R})$, 
one has $A_{ij}=a_{ij}$, where $A_{ij}=(-1)^{i+j}D_{ij}$ and $D_{ij}$
is a minor obtained by deleting the row $i$ and column $j$ from the
matrix $A.$ \label{thm:Aij}
\end{theorem}
\begin{proof}
It is well-known that for any invertible matrix,
$A^{-1}=\frac{1}{\det A}[A_{ij}]^T.$ In our case $\det A=1$ and
$A^{-1}=A^T$, which proves the claim.
\end{proof}

\begin{lemma}
\label{lm:3}
Let $A=[a_{ij}]\in SO_3(\mR)$.
Let $i,j,k$ be three different indices between $1$ and $3$. 
Then
\begin{eqnarray*}
(1+a_{ii})(a_{jk}+a_{kj})&=&a_{ij}a_{ki}+a_{ji}a_{ik},\\
(a_{jj}+a_{kk})(a_{jk}+a_{kj})&=&-(a_{ij}a_{ik}+a_{ji}a_{ki}),\\
(a_{ij}^2+a_{ik}^2)(a_{ij}a_{ik}+a_{ji}a_{ki})
&=&(a_{ij}a_{ki}+a_{ji}a_{ik})(a_{ij}a_{ji}+a_{ik}a_{ki}).
\end{eqnarray*}
\end{lemma}
\begin{proof}
By symmetry, it is sufficient to consider the case $i=1,j=2,k=3$ 
only. Using the previous theorem we have:
\begin{eqnarray*}
a_{23}+a_{32}
&=&
A_{23}+A_{32}=-(a_{11}a_{32}-a_{12}a_{31})-(a_{11}a_{23}-a_{21}a_{13})\\
&=&
-a_{11}(a_{23}+a_{32})+ a_{12}a_{31}+a_{21}a_{13}.
\end{eqnarray*}
Consequently, $ (1+a_{11})(a_{23}+a_{32})=a_{12}a_{31}+a_{21}a_{13}.$

The second equality follows from the orthogonality:
$$
(a_{22}+a_{33})(a_{23}+a_{32})\!
= \!(a_{22}a_{23} +a_{33}a_{33})+(a_{22}a_{32}+a_{23}a_{33})
=-a_{21}a_{31} -a_{12}a_{13}
$$
and we are done.

For the last equality we write:
\begin{eqnarray*}
&&
(a_{12}^2+a_{13}^2)(a_{13}a_{12}+a_{31}a_{21})
=
(a_{12}a_{31}+a_{21}a_{13})(a_{12}a_{21}+a_{13}a_{31})
\\
&\Leftrightarrow&
a_{12}^3a_{13} +a_{12}^2a_{31}a_{21}+a_{13}^3a_{12} +a_{13}^2a_{31}a_{21}
\\
&& \phantom{owimbaway}=
a_{12}^2a_{31}a_{21}+a_{21}^2a_{12}a_{13}
+
a_{31}^2a_{13}a_{12}+a_{13}^2a_{31}a_{21}
\\
&\Leftrightarrow&
a_{12}^3a_{13} +a_{13}^3a_{12} 
=
a_{21}^2a_{12}a_{13}+a_{31}^2a_{13}a_{12}
\\
&\Leftrightarrow&
a_{12}a_{13}(a_{12}^2+a_{13}^2)
=
a_{12}a_{13}(a_{21}^2+a_{31}^2)\\
&
\Leftrightarrow&
a_{12}a_{13}(1-a_{11}^2)=a_{12}a_{13}(1-a_{11}^2),
\end{eqnarray*}
where we used the orthogonality conditions.
\end{proof}

Now we are ready for the \textbf{first proof} of Theorem
\ref{thm:main}.

\begin{proof} We have
$$
AV=\left[\begin{array}{ccc}
   \displaystyle \frac{a_{11}}{a_{23}+a_{32}} 
   +\frac{a_{12}}{a_{13}+a_{31}}
   +\frac{a_{13}}{a_{12}+a_{21}}\\[0.3cm]
   \displaystyle \frac{a_{21}}{a_{23}+a_{32}}
   +\frac{a_{22}}{a_{13}+a_{31}} 
   +\frac{a_{23}}{a_{12}+a_{21}}\\[0.3cm]
   \displaystyle \frac{a_{31}}{a_{23}+a_{32}}
   +\frac{a_{32}}{a_{13}+a_{31}} 
   +\frac{a_{33}}{a_{12}+a_{21}}\\[0.3cm]
   \end{array}\right].
$$
We want to prove that
$$
\frac{a_{11}}{a_{23}+a_{32}}
+\frac{a_{12}}{a_{13}+a_{31}} 
+\frac{a_{13}}{a_{12}+a_{21}}
=
\frac{1}{a_{23}+a_{32}}
$$
(the proofs for other coordinates are similar).  Suppose first that
$a_{11}+1\neq 0.$ Then this is equivalent to
$$
\frac{(1-a_{11})(1+a_{11})}{(1+a_{11})(a_{23}+a_{32})}
=
\frac{a_{12}}{a_{13}+a_{31}} +\frac{a_{13}}{a_{12}+a_{21}}.
$$
By Lemma \ref{lm:3} this transforms to
\begin{eqnarray*}
&& 
\frac{1-a_{11}^2}{a_{12}a_{31}+a_{21}a_{13}}
=
\frac{a_{12}^2+a_{13}^2+a_{12}a_{21}+a_{13}a_{31}}{(a_{13}+a_{31})(a_{12}+a_{21})}
\\
&\Leftrightarrow&
 (a_{12}^2+a_{13}^2)\left(\frac{1}{a_{12}a_{31}+a_{21}a_{13}}
-\frac{1}{(a_{13}+a_{31})(a_{12}+a_{21})}\right)
\\
&&
\phantom{owimbawayowimbaway}
=
\frac{a_{12}a_{21}+a_{13}a_{31}}{(a_{13}+a_{31})(a_{12}+a_{21})}
\\
&\Leftrightarrow&
\frac{(a_{12}^2+a_{13}^2)(a_{13}a_{12}+a_{31}a_{21})}{({a_{12}a_{31}
+a_{21}a_{13})(a_{13}+a_{31})(a_{12}+a_{21})}}
=
\frac{a_{12}a_{21}+a_{13}a_{31}}{(a_{13}+a_{31})(a_{12}+a_{21})}
\\
&\Leftrightarrow&
(a_{12}^2+a_{13}^2)(a_{13}a_{12}+a_{31}a_{21})
=
(a_{12}a_{31}+a_{21}a_{13})(a_{12}a_{21}+a_{13}a_{31})
\end{eqnarray*}
and we can apply Lemma \ref{lm:3} again.

It remains to consider the case $a_{11}=-1.$ But then
$$
a_{12}^2+a_{13}^2=1-a_{11}^2=0\Rightarrow a_{12}=a_{13}=0.
$$
Similarly we get $a_{21}=a_{31}=0.$ But this contradicts
$a_{12}+a_{21}\neq 0.$
\end{proof}

So straightforward calculations was not so obvious as expected. We can
slightly improve them in our \textbf{second proof}.
 
\begin{proof}
If we apply Theorem \ref{thm:rank2} to the matrix $A-I$ which has 
rank $2$ we get the eigenvector directly.  Suppose that this is for
example
\begin{eqnarray*}
V_1
&=&
\left[
 \left|\begin{array}{cc}a_{22}-1 & a_{23} \\a_{32} & a_{33}-1 \end{array}\right|,\ 
-\left|\begin{array}{cc}a_{21} & a_{23} \\a_{31} & a_{33}-1 \end{array}\right|,\ 
 \left|\begin{array}{cc}a_{21} & a_{22}-1 \\a_{31} & a_{32} \end{array}\right|
\right]^T
\\
&=&
\left[A_{11} +1-a_{22}-a_{33},\ A_{12}+a_{21},\ A_{13}+a_{31}\right]^T
\phantom{\displaystyle \frac{1}{a_{32}}}
\\
&=&
\left[1+a_{11} -a_{22}-a_{33},\ a_{12}+a_{21},\ a_{13}+a_{31}\right]^T
\phantom{\displaystyle \frac{1}{a_{32}}}
\end{eqnarray*}
obtaining the vector $W_1$ from Theorem \ref{thm:general}, so we get
part of this theorem as well.  Vectors $V_2,V_3$ lead us naturally to
$W_2,W_3.$ To finish the proof of Theorem \ref{thm:main}, we divide
the obtained vector by $(a_{12}+a_{21})(a_{13}+a_{31})$ (which is
non-zero), and it remains to show that
$$
\frac{1+a_{11} -a_{22}-a_{33}}{(a_{12}+a_{21})(a_{13}+a_{31})}
=
\frac{1}{a_{23}+a_{32}}.
$$
By Lemma \ref{lm:3} we have
\begin{eqnarray*}
&&(1+a_{11} -a_{22}-a_{33})({a_{23}+a_{32}})\\
&&\phantom{wumbawayway} 
= ( 1+a_{11})({a_{23}+a_{32}}) -(a_{22}+a_{33})({a_{23}+a_{32}})\\
&&\phantom{wumbawayway} 
= a_{12}a_{31}+a_{21}a_{13}+a_{21}a_{31} +a_{12}a_{13}\\
&&\phantom{wumbawayway} 
=(a_{12}+a_{21})(a_{13}+a_{31}),
\end{eqnarray*}
which finishes the proof.
\end{proof}

\section{Origin of the non-trivial eigenvector}

Now we want to understand the origin of this non-trivial eigenvector.
We find one possible source in skew-symmetric matrices.

\begin{theorem} \label{thm:A2}
Let $A$ be an orthogonal matrix $($of any size$)$. If $U\in \ker( A-A^T)$, then 
$A^2U=U.$ Moreover, if $A$ has only one real eigenvalue $\lambda$,
then $AU=\lambda U$.
\label{Thm:skew}
\end{theorem}
\begin{proof}
We have
$$
(A-A^T)U=0\Leftrightarrow AU=A^TU\Leftrightarrow A^2U=U,
$$
which proves the first statement.

Let $\{e_i\}$ be a (complex) basis of eigenvectors (which exists
because $A$ is a normal matrix). If $U=\sum x_ie_i$, then
$$
A^2U-U=\sum x_i(\lambda_i^2-1)e_i=0,
$$
which means that all $x_i$ corresponding to complex eigenvalues
$\lambda_i$ should be equal to zero and $U$ is proportional to the
only eigenvector with real eigenvalue.
\end{proof}

\noindent Now we are ready for the \textbf{third proof}
of Theorem \ref {thm:main}.

\begin{proof}
Suppose first that $A\neq A^T$, that is, $A^2\neq I.$ Then $A$ has 
some complex eigenvalue $\lambda.$ It follows that
$\overline{\lambda}$ is another eigenvalue, and the third one is $1$
(because $|\lambda|=1$ and $\det A=1$).  Since
$$
U=\left[\begin{array}{c}
   a_{23}-a_{32} \\
   a_{31}-a_{13} \\
   a_{12}-a_{21}
  \end{array} \right]\in \ker (A-A^T),
$$
by Theorem~\ref{thm:skew}, and is a non-zero vector, we can apply
Theorem \ref{Thm:skew} to get $AU=U$.  We need only to show that
$cV=U$ for some non-zero $c.$ We put $c=a_{23}^2-a_{32}^2$, and note
that $c=a_{31}^2-a_{13}^2$, $c=a_{12}^2-a_{21}^2$ as well, for example
$$
a_{23}^2-a_{32}^2=a_{31}^2-a_{13}^2
\Leftrightarrow
a_{13}^2+a_{23}^2=a_{31}^2+a_{32}^2
\Leftrightarrow
1-a_{33}^2=1-a_{33}^2.
$$
Then
\begin{eqnarray*}
cV
&=&
\left[\begin{array}{ccc} 
\displaystyle \frac{c}{a_{23}+a_{32}}&  
\displaystyle\frac{c}{a_{13}+a_{31}}& 
\displaystyle\frac{c}{a_{12}+a_{21}}
\end{array}\right]^T
\\
&=&
\left[\begin{array}{ccc} 
\displaystyle\frac{a_{23}^2-a_{32}^2}{a_{23}+a_{32}}& 
\displaystyle\frac{a_{31}^2-a_{13}^2}{a_{13}+a_{31}}&
\displaystyle\frac{a_{12}^2-a_{21}^2}{a_{12}+a_{21}}
\end{array}\right]^T=U.
\end{eqnarray*}
It remains to consider the case $A=A^T,$ that is, $a_{ij}=a_{ji}$, and
we need to prove that for
$$
V'
=
\left[\begin{array}{ccc} 
\displaystyle \frac{1}{a_{23}}&  
\displaystyle\frac{1}{a_{13}}& 
\displaystyle\frac{1}{a_{12}}
\end{array}\right]^T,
$$
we have $AV'=V'.$ This can be done explicitly, for example for the
first coordinate we have
$$
\frac{a_{11}}{a_{23}}+  \frac{a_{12}}{a_{13}} +\frac{a_{13}}{a_{12}}
=
\frac{1}{a_{23}}
\Leftrightarrow 
\frac{a_{12}^2+a_{13}^2}{{a_{12}a_{13}}} 
=
\frac{1-a_{11}}{a_{23}}
\Leftrightarrow
\frac{1-a_{11}^2}{{a_{12}a_{13}}} 
= 
\frac{1-a_{11}}{a_{23}}
$$
So we need only to prove
$$
(1+a_{11})a_{23}=  a_{12}a_{13}
\Leftrightarrow 
a_{23}= a_{12}a_{31}-a_{11}a_{32}
\Leftrightarrow 
a_{23}=A_{23},
$$
which follows from Theorem \ref{thm:Aij}. Note also that we completed
the proof of Theorem \ref{thm:general} regarding the vector
$U.$\end{proof}

\section{A geometric interpretation of the eigenvector}

Now we want to find some geometrical interpretation of our eigenvector
and consider \textbf{fourth proof} of Theorem \ref {thm:main}.

\begin{proof}
The starting point is that any matrix $A\in SO_3(\mathbb{R})$ can be 
written as a product of two reflections. (This is easy to see in the
plane, and as every rotation in $\mR^3$ has an axis of rotation, the
result for rotations in $\mR^3$ follows from the planar case.) So let
$X,Y$ be two unit vectors such that $A=(I-2XX^T)(I-2YY^T).$ The case
when $X$ and $Y$ are proportional is not interesting for us (in this
case $A=I$). So we suppose that they are linear independent and let
$Z=X\times Y$ be their (nonzero) vector product. First we note that
$Z$ is the eigenvector we are looking for. Indeed,
$X^TZ=\langle X,Z\rangle=0$ and similarly $Y^TZ=0,$ giving
$AZ=(I+BX^T+CY^T)Z=IZ=Z.$ As we know that
$$
Z=
\left[\begin{array}{ccc} 
x_2y_3-x_3y_2 & x_3y_1-x_1y_3 & x_1y_2-x_2y_1
\end{array}\right]^T,
$$
we need only to prove that our vector $v$ is proportional to this one,
that is,
$$
\det 
\left[\begin{array}{cc} 
v_i   & z_i  \\ v_j &  z_j 
\end{array} \right]
=0. 
$$
By symmetry, it is sufficient to consider the case $i=1,j=2$ only. We
have
\begin{eqnarray*}
&&
\det 
\left[\begin{array}{cc} 
\displaystyle \frac{1}{a_{23}+a_{32}} &x_2y_3-x_3y_2\\[0.3cm]
\displaystyle \frac{1}{a_{13}+a_{31}} &x_3y_1-x_1y_3\\[0.3cm]
\end{array}\right]=0\\
&\Leftrightarrow &
(x_3y_1-x_1y_3)(a_{13}+a_{31})=(x_2y_3-x_3y_2)(a_{23}+a_{32}).
\end{eqnarray*}
Let $c=\langle X,Y\rangle .$ Then $A=I-2XX^T-2YY^T+4cXY^T$, and for
$i\neq j$,
$$
a_{ij}+a_{ji}=-4x_ix_j-4y_iy_j+4c(x_iy_j+x_jy_i).
$$
Our aim is
\begin{eqnarray*}
&&
(x_3y_1-x_1y_3)(-x_1x_3-y_1y_3)+c(x_3y_1+x_1y_3)
\\
&&\phantom{wimba}
=
(x_2y_3-x_3y_2)(-x_2x_3-y_2y_3)+c(x_2y_3+x_3y_2)\\
&\Leftrightarrow&
x_1y_1(-x_3^2+y_3^2)+x_3y_3(-y_1^2+x_1^2)+c((x_3y_1)^2-(x_1y_3)^2)\\
&&\phantom{wimba}
=
x_3y_3(-x_2^2+y_2^2)+x_2y_2(-y_3^2+x_3^2)+c((x_2y_3)^2-(x_3y_2)^2)\\
& \Leftrightarrow&
(x_1y_1+x_2y_2)(-x_3^2+y_3^2)+x_3y_3(-y_1^2+x_1^2+x_2^2-y_2^2)\\
&&\phantom{wimba}=
c(y_3^2(x_1^2+x_2^2))-x_3^2(y_1^2+y_2^2)).
\end{eqnarray*}
Now we use the fact that we have unit vectors.
\begin{eqnarray*}
&& (x_1y_1+x_2y_2)(-x_3^2+y_3^2)+x_3y_3(1+y_3^2-1- x_3^2)\\
&& \phantom{wimba}=c(y_3^2(1-x_3^2))-x_3^2(1-y_3^2))\\
&\Leftrightarrow&
(x_1y_1+x_2y_2+x_3y_3)(-x_3^2+y_3^2)=c(y_3^2-x_3^2)
\end{eqnarray*}
and we are done because $c=x_1y_1+x_2y_2+x_3y_3.$
\end{proof}

\section{A proof using the Lie algebra of the rotation group}

Define the Lie algebra
$$
\mathfrak{so}_3(\mR):=\{Q\in \mR^{3\times 3}: Q+Q^T=0\}
$$
of the Lie group $SO_3(\mR)$. We recall the following well-known
result; see for example \cite[Lemma 1B,p.31]{Ros}.

\begin{proposition}
\label{prop_exp}
Let $A\in SO_3(\mR)$. Then there exists a $t\in[0,2\pi)$ and a matrix
$Q\in \mathfrak{so}_3(\mR)$ such that $A=e^{tQ}$. Moreover, defining
$U=[p\;\;q\;\;r]^T\in \mR^3$ by
$$
Q=\left[\begin{array}{rrr}
  0 & -r & q \\
  r & 0 & -p \\
 -q & p & 0
  \end{array}\right],
$$
$A$ is a rotation about $U$ through the angle $t$ using the right-hand
rule.
\end{proposition}

\noindent
We will also need the fact that for $t\geq 0$,
$$
e^{t Q}=\calL^{-1}((sI-Q)^{-1})(t),
$$
where $\calL^{-1}$ denotes the (entrywise) inverse one-sided Laplace
transform.  The following fact is well-known (see for example,
\cite[\S27,p.218]{Bel}):

\begin{proposition}
\label{theorem_LT_of_expA}
For large enough $s$,
$\displaystyle \int_{0}^{\infty}e^{-st}e^{tQ}dt=(sI-Q)^{-1}$.
\end{proposition}

\noindent In the above, the integral of a matrix whose elements are
functions of $t$ is defined entrywise.  If $s$ is not an eigenvalue of
$Q$, then $sI-Q$ is invertible, and by Cramer's rule,
$$
(sI-Q)^{-1}=\frac{1}{\det(sI-Q)} \textrm{adj}(sI-Q).
$$
So we see that each entry of $\textrm{adj}(sI-Q)$ is a polynomial in
$s$ whose degree is at most $n-1$, where $n$ denotes the size of $Q$,
that is, $Q$ is an $n\times n$ matrix. Consequently, each entry
$m_{ij}$ of $(sI-Q)^{-1}$ is a rational function in $s$, whose inverse
Laplace transform gives the matrix exponential $e^{tQ}$.  We now give
the \textbf{fifth proof} of Theorem~\ref{thm:main}.

\begin{proof}
Let $Q,U$ be as in Proposition~\ref{prop_exp}. By Cramer's rule,
\begin{eqnarray*}
(sI-Q)^{-1}
&=&
\left[\begin{array}{rrr}
s & r & -q \\ -r & s & p \\ q & -p & s
\end{array}\right]^{-1}
\\
&=&
\frac{1}{\det(sI-Q)}
\left[\begin{array}{ccc}
 s^2+p^2  & rs+pq & -qs+rp \\
 -rs+pq & s^2+q^2 & ps +qr \\
 qs+rp & -ps+qr & s^2+r^2
\end{array}\right].
\end{eqnarray*}
Hence
$$
A=e^{tQ}=
\calL^{-1}\left(\frac{1}{\det(sI-Q)}
\left[\begin{array}{ccc}
 s^2+p^2  & rs+pq & -qs+rp \\
 -rs+pq & s^2+q^2 & ps +qr \\
 qs+rp & -ps+qr & s^2+r^2
 \end{array}\right]\right)(t).
$$
This yields
$$
V\!=\!
\left[\begin{array}{c}
\displaystyle \frac{1}{a_{23}+a_{32}}\\[0.3cm]
\displaystyle\frac{1}{a_{31}+a_{13}}\\[0.3cm]
\displaystyle\frac{1}{a_{12}+a_{21}}\\[0.3cm]
\end{array}\right]
\!=\!  
\underbrace{\left(\calL^{-1}
\left(\frac{1}{\det(sI-Q)}\right)(t)\right)^{-1}}_{=:c}
\left[\begin{array}{c}
\displaystyle\frac{1}{2qr}\\[0.3cm]
\displaystyle\frac{1}{2rp}\\[0.3cm]
\displaystyle\frac{1}{2pq}\\[0.3cm]
\end{array}\right]
\!=\!
\frac{c}{2pqr}
\left[\begin{array}{c}
       p\\
       q\\
       r
      \end{array}
\right],
$$
which is a multiple of $U$.
\end{proof}

\section{A quaternionic proof}

Let $ \bbD:=\{\bbq=a+b\bbi+c\bbj +d \bbk:a,b,c,d\in \mR\} $ be the
ring of all quaternions, with $\bbi^2=\bbj^2=\bbk^2=-1$ and
$\bbi\cdot \bbj=-\bbj\cdot \bbi=\bbk$,
$\bbj\cdot \bbk=-\bbk\cdot \bbj=\bbi$,
$\bbk\cdot \bbi=-\bbi\cdot \bbk=\bbj$.  
We define the norm of $\bbq=a+b\bbi+c\bbj +d\bbk$ by
$$
|\bbq|=\sqrt{a^2+b^2+c^2+d^2},
$$
and the conjugate $\overline{\bbq}$ of $\bbq$ by
$$
\overline{\bbq}=a-b\bbi-c\bbj -d\bbk.
$$
It can be checked that for $\bbq_1,\bbq_2\in \bbD$,
$|\bbq_1 \bbq_2|=|\bbq_1||\bbq_2|$ and
$|\bbq|^2=\bbq \overline{\bbq}$.  We identify $\mR^3$ as a subset of
$\bbD$ via
$$
\mR^3=\{b\bbi +c\bbj+d\bbk\in \bbD: b,c,d\in \mR\}.
$$
If $|\bbq|=1$ then for any $\bbw\in \mR^3$,
$\bbq\bbw \bbq^{-1}\in \mR^3$, for example
\begin{eqnarray*}
\bbq\bbi \bbq^{-1}=\bbq\bbi\overline{\bbq}
&=&
(a+b\bbi+c\bbj +d\bbk)\bbi(a-b\bbi-c\bbj -d\bbk)\\
&=&
(a\bbi-b-c\bbk+d\bbj)(a-b\bbi-c\bbj -d\bbk)\\
&=&
a^2\bbi +ab-ac\bbk+ad\bbj -ba+b^2\bbi+bc\bbj +bd\bbk\\
&&-ca\bbk+cb\bbj-c^2\bbi-cd +da\bbj+db\bbk+dc-d^2\bbi\\
&=&
(a^2+b^2-c^2-d^2)\bbi+2(ad+bc)\bbj+2(bd-ac)\in \mR^3.
\end{eqnarray*}
So the map $T_\bbq:\bbw\mapsto \bbq\bbw \bbq^{-1}$ maps vectors in
$\mR^3$ to vectors in $\mR^3$ and clearly is linear.  In fact, this
collection of maps $T_\bbq$, $|\bbq|=1,$ is precisely the set $SO(3)$
of rotations in $\mR^3$!

To see this note first that if $\bbw\in \mR^3$, then its Euclidean
norm $\|\bbw\|_2$ coincides with its quaternionic norm.  Therefore
$T_\bbq$ is also a rigid motion, since
$$
\|T_\bbq \bbw\|_2
=
|T_\bbq \bbw|
=
|\bbq \bbw \bbq^{-1}|
=
|\bbq| |\bbw| |\bbq^{-1}|
=
|\bbw|
=
\|\bbw\|_2
$$ 
so our map corresponds to an orthogonal matrix. But because
$$
T_\bbq(\bbq-a)=\bbq(\bbq-a)\bbq^{-1}=\bbq^2\bbq^{-1}-a\bbq\bbq^{1}=\bbq-a
$$ 
we have an invariant vector as well (when $\bbq=a$ we can take any
vector), so our matrix belongs to $SO(3)$ and is a rotation. We can
describe it explicitly.

Since $|a|\leq 1$, we can find a unique $t\in [0,2\pi)$ such that
$\displaystyle \cos \frac{t}{2}=a $ to get
$$
\bbq= \left(\cos \frac{t}{2}\right)+\bbv.
$$
We leave to the reader to prove that the angle of rotation around
$\bbv$ is exactly $t$.  It is clear that every rotation then arises in
this manner.

Now we are ready to give the \textbf{sixth proof} of
Theorem~\ref{thm:main}.

\begin{proof} 
We need to consider the case $\bbv\neq 0$ only. By feeding in 
$\bbi,\bbj,\bbk$ into $T_\bbq$, we can now compute the matrix $A$  of
$T_\bbq$ in terms of the entries of $[b\;\; c \;\;d]^T$, where
$\bbv=b\bbi+c\bbj+d\bbk$.  We already know the first column and the
rest we get by cyclic symmetry:
$$
A=\left[ \begin{array}{ccc}
          a^2+b^2-c^2-d^2 & 2(bc-ad) & 2(ac+bd) \\
         2(bc+ad) & a^2+c^2-b^2-d^2  & 2(cd-ab)\\
         2(bd-ac)  & 2(ab+cd) & a^2+d^2-b^2-c^2
  \end{array}\right],
$$
Now it is easy to check that
$$
V=
\left[\begin{array}{c}
\displaystyle \frac{1}{a_{23}+a_{32}}\\[0.3cm]
\displaystyle\frac{1}{a_{31}+a_{13}}\\[0.3cm]
\displaystyle\frac{1}{a_{12}+a_{21}}\\[0.3cm]
\end{array}\right]
=
\left[\begin{array}{c}
\displaystyle \frac{1}{4cd}\\[0.3cm]
\displaystyle\frac{1}{4bd}\\[0.3cm]
\displaystyle\frac{1}{4bc}\\[0.3cm]
\end{array}\right]
=
\frac{1}{4bcd}
\left[\begin{array}{c}
       b\\
       c\\
       d
\end{array}\right]
$$
which is a multiple of $\bbv$.
\end{proof}

\section{A proof using the Cayley transform}

We only consider the case when $-1$ is not eigenvalue of $A$, since
the case when $-1$ is an eigenvalue of $A$ (implying that $A^2=I$) has
been covered before in our third proof.

\begin{theorem}
If $A\in SO_3(\mR)$ such that $-1$ is not an eigenvalue of $A$, 
then there exists a skew-symmetric $Q$ such that $A=(I+Q)(I-Q)^{-1}$.
\end{theorem}
\begin{proof} As $-1$ is not an eigenvalue of $A$, $A+I$ is 
invertible. Define
$$
Q=(A-I)(A+I)^{-1}.
$$
Then 
\begin{eqnarray*}
 Q+Q^T&=&(A-I)(A+I)^{-1} +(A^T+I)^{-1}(A^T-I)\\
 &=&(A-I)(A+I)^{-1} +(A^{-1}+I)^{-1}(A^{-1}-I)\\
 &=&(A-I)(A+I)^{-1} +(I+A)^{-1}AA^{-1} (I-A)\\
 &=&(A-I)(A+I)^{-1} +(I+A)^{-1} (I-A)=0,
\end{eqnarray*}
where we use the commutativity to get the last equality.  So $Q$ is
skew-symmetric. But then $I-Q$ is invertible.  From the definition of
$Q$, it follows that $Q(A+I)=A-I$, and solving for $A$, we obtain
$A=(I+Q)(I-Q)^{-1}$.
\end{proof}

Now we are ready to give the \textbf{seventh proof} of
Theorem~\ref{thm:main}.

\begin{proof} Given $A$, we can write $A$ as $A=(I+Q)(I-Q)^{-1}$ 
for some skew-symmetric $Q$
$$
Q=
\left[\begin{array}{rrr}
0 & -r & q \\ r & 0 & -p \\ -q & p & 0
\end{array}\right].
$$
Then 
\begin{eqnarray*}
A\!\!\!
&=&\!\!\!\!(I+Q)(I-Q)^{-1}\\
&=& \!\!\!\!
\frac{1}{1\!+\!p^2\!+\!q^2\!+\!r^2}\!
\left[\!\!\!\begin{array}{ccc} 
  1\!+\!p^2\!-\!q^2\!-\!r^2 & 2pq-2r & 2rp+2q \\
  2pq+2r & 1\!-\!p^2\!+\!q^2\!-\!r^2 & 2qr-2p \\
  2rp-2q & 2qr+2p & 1\!-\!p^2\!-\!q^2\!+\!r^2 
\end{array}\right],
\end{eqnarray*}
and
$$
\left[\begin{array}{c}
       \displaystyle \frac{1}{a_{23}+a_{32}}\\[0.3cm]
       \displaystyle\frac{1}{a_{31}+a_{13}}\\[0.3cm]
       \displaystyle\frac{1}{a_{12}+a_{21}}\\[0.3cm]
\end{array}\right]
=
(1\!+\!p^2\!+\!q^2\!+\!r^2)
\left[\begin{array}{c}
       \displaystyle \frac{1}{4qr}\\[0.3cm]
       \displaystyle\frac{1}{4rp}\\[0.3cm]
       \displaystyle\frac{1}{4pq}\\[0.3cm]
\end{array}\right]
=
\frac{1\!+\!p^2\!+\!q^2\!+\!r^2}{4pqr}
\left[\begin{array}{c}
       p\\
       q\\
       r
\end{array}\right]
$$
which is an eigenvector of $A$ corresponding to eigenvalue $1$, by
Theorem~\ref{thm:skew}.
\end{proof}

\section{A proof using contour integral of the resolvent}

We recall the following; see for example \cite[\S8.2, p.127]{God}:

\begin{proposition} 
For an isolated eigenvalue of a square matrix $A$, enclosed inside a 
simple closed curve $\gamma$ running in the anti-clockwise direction,
the projection $P$ onto the eigenspace $\ker (\lambda I-A)$ is given
by
$$
P=\frac{1}{2\pi i} \oint_\gamma (zI-A)^{-1} dz.
$$
\end{proposition}

\noindent We are now ready to give the {\bf eighth proof} of
Theorem~\ref{thm:main}.

\begin{proof}
Let $A\in SO_3(\mR)$. Again we restrict ourselves to the case that 
$A\neq I$. Then we have that $1$ is an isolated simple eigenvalue. Let
the other two eigenvalues be denoted by $\lambda,\overline{\lambda}$,
and let $p_{ij}(z)$ be the minor obtained by deleting the row $i$ and
column $j$ from the matrix $zI-A$. If $\gamma$ encloses $1$, but not
the other two eigenvalues $\lambda,\overline{\lambda}$, then we have
\begin{eqnarray*}
P
&=&
\frac{1}{2\pi i} \oint_\gamma (zI-A)^{-1} dz=
\frac{1}{2\pi i} \oint_\gamma \frac{1}{\det (zI-A)} [p_{ij}(z)] dz
\\
&=&
\frac{1}{2\pi i} 
\oint_\gamma \frac{1}{(z-1)(z-\lambda)(z-\overline{\lambda})} [p_{ij}(z)] dz
=
\frac{1}{(1-\lambda)(1-\overline{\lambda})} [p_{ij}(1)],
\end{eqnarray*}
where we have used the Cauchy Integral Formula \cite[Cor.3.5,
p.94]{Sas} to obtain the last equality.  In particular
\begin{eqnarray*}
P\left[\begin{array}{c} 1 \\ 0 \\ 0\end{array}\right]
&=&
\frac{1}{|1-\lambda|^2}  
\left[\begin{array}{c}
(1-a_{22})(1-a_{33})-a_{23}a_{32} \\ 
a_{12}(1-a_{33}) +a_{13} a_{32} \\
a_{12} a_{23} +a_{13}(1-a_{22}) \\  
\end{array}\right]
\\
&=&\frac{1}{|1-\lambda|^2}
\left[\begin{array}{c} 
1-a_{22}-a_{33}+A_{11} \\a_{12}+A_{21} \\ a_{13}+A_{31}
\end{array}\right]
\\
&=&\frac{1}{|1-\lambda|^2}
\left[\begin{array}{c} 
1-a_{22}-a_{33}+a_{11} \\a_{12}+a_{21} \\ a_{13}+a_{31}
\end{array}\right]
\\
&=& \frac{1}{|1-\lambda|^2}
\left[\begin{array}{c} 
\frac{(a_{12}+a_{21})(a_{13}+a_{31})}{a_{23}+a_{32}} \\ a_{12}+a_{21} \\ a_{13}+a_{31}
\end{array}\right]
=c
\left[\begin{array}{c} 
\frac{1}{a_{23}+a_{32}} \\[0.1cm] 
\frac{1}{a_{13}+a_{31}} \\[0.1cm] 
\frac{1}{a_{12}+a_{21}}\\[0.1cm]
\end{array}\right],
\end{eqnarray*}
for some constant $c$. 
\end{proof}

Note that we recover the vector $W_1$ from
Theorem~\ref{thm:general}. $W_2,W_3$ can be found similarly.

\section{What about zeros?}

Now it is time to think about the conditions $a_{ij}+a_{ji}\neq 0.$
What if some of them failed e.g. $a_{12}+a_{21}=0?$ The eigenvector
still exists, but how does it look now?  Note first
that
$$
a_{13}^2=1-a_{11}^2-a_{12}^2=1-a_{11}^2-a_{21}^2=a_{31}^2
\Rightarrow 
a_{13}=\pm a_{31}.
$$
Similarly $a_{23}=\pm a_{32}.$ So our matrix looks now as
$$
\left[\begin{array}{rrr}
      a & r & q \\
      -r & b & p \\
      \varepsilon q & \zeta p & c 
\end{array}\right],
$$
where $\varepsilon^2=\zeta^2=1.$ Suppose first that $pqr\neq 0$. The
orthogonality conditions for the first two rows gives:
$$
-ar+br+pq=0\Leftrightarrow pq=r(a-b).
$$  
For the first two columns we get instead
$$
ar-br+\varepsilon\zeta pq=0
\Leftrightarrow 
\varepsilon\zeta pq=r(b-a)
$$ 
thus $\varepsilon\zeta =-1\Leftrightarrow \zeta=-\varepsilon.$

Now  for $\varepsilon=-1$ we simply put $V=[0\;\;q\;\;-r]^T.$  
We have
$$
AV=
\left[\begin{array}{rrr}
      a & r & q \\
      -r & b & p \\
      -q &  p &  c 
\end{array}\right]
\left[\begin{array}{r}
       0 \\
       q \\
       -r 
\end{array}\right]
=
\left[\begin{array}{c}
 0 \\
 bq-rp \\
 pq-cr  
\end{array} \right]
=
\left[\begin{array}{r}
      0 \\
     q \\
     - r 
\end{array} \right],
$$
where the last equality follows from Theorem \ref{thm:Aij}.

If $\varepsilon=1$ we take instead $V=[p\;\;0\;\; r]^T$ with similar
argument:
$$
AV
=
\left[\begin{array}{rrc}
      a & r & q \\
      -r & b & p \\
      q &  -p &  c 
\end{array}\right]
\left[\begin{array}{r}
       p \\
       0 \\
       r 
\end{array}\right]
=
\left[\begin{array}{c}
       ap+qr \\
         0 \\
       pq+cr   
\end{array}\right]
=
\left[\begin{array}{r}
       p \\
       0 \\
       r 
\end{array}\right],
$$
where again the last equality follows from Theorem \ref{thm:Aij}.

Thus the rule is easy: for exactly one pair of indices $i,j$ we have
$a_{ij}=a_{ji}.$ If $k$ is the remaining index put
$v_k=0, v_i=a_{k,j},v_j=-a_{k,i}.$

In fact we can describe matrices above almost explicitly.  To make
calculations more homogeneous we put $c=\varepsilon d$ as well.
Consider the remaining orthogonal conditions for different rows:
$$
\varepsilon aq -\varepsilon pr +\varepsilon dq=0 \Leftrightarrow pr=q(a+d),
$$
$$
-\varepsilon pq-\varepsilon br+\varepsilon dr=0\Leftrightarrow pq=r(-b+d).
$$  
Pairwise multiplications of the obtained equations and cancelling
gives:
$$
r^2=(a-b)(a+d);\ q^2=(a-b)(-b+d);\ p^2=(a+d)(-b+d).
$$
Now the last orthogonality condition is
\begin{eqnarray*}
1&=&a^2+p^2+q^2=a^2+(-b+d)(2a-b+d)\\
 &=&a^2+2a(-b+d)+(-b+d)^2=(a-b+d)^2,
\end{eqnarray*}
or $a-b+d=\pm 1$ (other rows and columns gives the same). Now we can
choose $a,b$ as parameters (with natural restrictions, e.g. $|a|< 1$)
and reconstruct the rest choosing signs. As example we get
$$
A=\frac{1}{3}
\left[\begin{array}{rrr}
  1 & 2 & 2 \\
 -2 & -1 & 2 \\ 
 -2 & 2 & -1 
\end{array}\right]
$$
It remains to consider the case $pqr=0.$ If for example $p=0$ then by
the orthogonality of two first rows $qr=0$ as well and similarly for
other cases we get that at least two of $p,q,r$ are zero. Then the
corresponding column containing them is an eigenvector directly.

\section{Possible generalisations}

So far we concentrated on $3\times 3$ real matrices, especially on the
case $A \in SO_3(\mathbb{R})$. But we now ask: what can be
generalised?  Theorem \ref{thm:main} is obviously valid for any
orthogonal matrix (that is why we have $A \in O_3(\mathbb{R})$ in the
abstract), and moreover, it is valid for any matrix $A=cA'$ with
$A' \in SO_3(\mathbb{R}).$ Theorem~\ref{thm:general} is valid as well
if we replace the constant $1$ in the vectors $W_i$ by $c\neq 0.$

For larger sizes, we still have the analogues of Theorem
\ref{thm:rank2} and Theorem~\ref{thm:Aij} and can imitate the
{second proof} to obtain the analogues of the vectors $W_i.$
But already for the size $5$ (where the vector $V$ with $AV=V$
exists), the expressions involve determinants of size $3$, and its is
hardly attractive to write them here. The vector $U$ obtained in the
{third proof} is also in principle available, but we have no
easy analogue of Theorem \ref{thm:skew}, while an analogue of Theorem
\ref{thm:rank2} produces the determinants of high order. And the idea
to generalise Theorem \ref{thm:main} to higher dimensions looks
hopeless.

What if we change the field? Because the conditions $A^{-1}=A^T$ and
$\det A=1$ are purely algebraic, all purely algebraic proofs survive,
and we have the same Theorem~\ref{thm:general} but we need some modifications.

First of all,  we should understand why $1$ is still an eigenvalue. This is easy. If
$\alpha, \beta, \gamma$ are our eigenvalues, then $\frac{1}{\alpha},\frac{1}{\beta},\frac{1}{\gamma}$ is the same set of numbers, but they may be in a different order. If for example,
$\frac{1}{\alpha}=\beta$ then $\alpha\beta=1$ and the condition $\det A =1$ gives $\gamma=1.$ The only remaining case is $\frac{1}{\alpha}=\alpha$, and then $\alpha=\pm 1$, and similarly for $\beta$ and $\gamma$,
but because their product is $1$ at least one of them is equal to $1$ as well. So the second proof survives completely, and the third need only an adjustment in the place where we used Theorem~\ref{thm:A2}.

The first proof has another weak point: for arbitrary field $x^2+y^2=0$ does not imply $x=y=0$ which we have used in the special case $a_{11}=-1.$ The case when $a_{12}\neq 0$ can really happen. Here is a  a nice example 
in $\mathbb{Z}_5$\ :
$$A=\left[\begin{array}{ccc}
 -1 & -1 & -2 \\  
 -2 & -1 & -1 \\
 -1 & -2 & -1
\end{array}\right].\ 
$$ But $A$ still have a correct eigenvector. The proof therefore should be modified (e.g. consider $i$ in our field such that $i^2=-1,$ write $a_{13}=ia_{12}$ and $a_{31}=\pm ia_{21}$ and continue in the same style as we have done
in the previous section to describe all possible exceptional matrices), but we prefer to skip this and restrict ourselves by only one algebraic proof).

 So the conditions $A^{-1}=A^T$ and
$\det A=1$ are sufficient to our main theorems.
The interesting
question is therefore: what is the class of the matrices that satisfy
those conditions? It is obviously a group. We study matrices of size
$2$ first.
$$
A= 
\left[\begin{array}{cc} a & b \\ c & d \end{array}\right]
\Rightarrow 
\left[\begin{array}{rr} d & -b \\ -c & a \end{array}\right]
= A^{-1} = A^T =
\left[\begin{array}{cc} a & c \\ b & d \end{array}\right],
$$
therefore $a=d$, $c=-b$, and $a^2+b^2=1.$ For the complex numbers, we
put $a=\cos z, b=\sin z$ for some complex number $z$ and get all the
solutions. So matrices such as
$$
\left[\begin{array}{ccc}
 1 & 0 & 0 \\  
 0 & \phantom{-}\cos z & \sin z \\
 0 & -\sin z  & \cos z 
\end{array}\right],\ 
\left[\begin{array}{ccc}
\phantom{-}\cos z & 0 & \sin z \\
 0 & 1 & 0  \\
-\sin z & 0  & \cos z 
\end{array}\right]
$$
and their products belongs to our group, so it is large enough. For
finite fields we can have difficulties to find ''cosines'' (for
example, in $\mathbb{Z}_5$, we have $a^2+b^2=1\Rightarrow a=0,b=1 $ or
$a=1,b=0$), but already in $\mathbb{Z}_7$ we have $2^2+2^2=1$ which
produces some matrices. But we prefer to skip this intriguing topic
for now.

Any time one gets a result about the orthogonal matrices, it is
natural to wonder about their complex relatives - unitary matrices.
What can be said about them? Most parts of the proofs fail, which is
not surprising, because now $A_{ij}=\overline{a_{ij}}$, and
skew-Hermitian matrix can be invertible, and can have non-zero
elements on the main diagonal. So we have no direct analogue of
Theorem \ref{thm:main}.  We can get some results if we know the
eigenvalue, but is nothing else than the direct application of Theorem
\ref{thm:rank2} (as in the {second proof}).

\begin{theorem} 
Let $A\in SU(3)$ be an unitary matrix with $($simple$)$ eigenvalue 
equal to $\lambda.$ Then for all the vectors
\begin{eqnarray*}
W_1&=&
\left[\begin{array}{ccc}
\displaystyle {\overline{a_{11}}+\lambda^2-\lambda(a_{22}+a_{33})}& 
\displaystyle {\overline{a_{12}}+a_{21}}& 
\displaystyle {\overline{a_{13}}+a_{31}}
\end{array}\right]^T,\\
W_2&=&
\left[\begin{array}{ccc}
\displaystyle {a_{12}+a_{21}}& 
\displaystyle {\overline{a_{22}}+\lambda^2-\lambda(a_{11}-a_{33})}& 
\displaystyle {a_{23}+a_{32}} 
\end{array}\right]^T,\\
W_3&=&
\left[\begin{array}{ccc}
\displaystyle {a_{13}+a_{31}}&  
\displaystyle {a_{23}+a_{32}}&
\displaystyle {\overline{a_{33}}+\lambda^2-\lambda(a_{11}-a_{22})} 
\end{array}\right]^T,
\end{eqnarray*}
we have $AW_i=\lambda W_i$, and at least one of them is non-zero, and
therefore is the eigenvector.
\end{theorem}

\end{document}